\documentclass{article}
\usepackage{amsmath, amsthm}

\newtheoremstyle{theorem}
{10pt} 
{10pt} 
{\sl} 
{\parindent} 
{\bf} 
{. } 
{ } 
{} 
\theoremstyle{theorem}
\newtheorem{theorem}{Theorem}
\newtheorem{corollary}[theorem]{Corollary}
\newtheorem{lemma}[theorem]{Lemma}
\newtheorem{exam}[theorem]{Example}
\newtheoremstyle{defi}
{10pt} 
{10pt} 
{\rm} 
{\parindent} 
{\bf} 
{. } 
{ } 
{} 
\theoremstyle{defi}
\newtheorem{definition}[theorem]{Definition}

\begin{document}

\title{\Large{\textbf{Some Inequalities for Nilpotent Multipliers of Finite Groups }}}
\author{\textbf{B. Mashayekhy, F. Mohammadzadeh, and A. Hokmabadi}\\
Department of Mathematics,\\ Center of Excellence in Analysis on Algebraic Structures,\\ Ferdowsi University of Mashhad,\\
P. O. Box 1159-91775, Mashhad, Iran.\\
mashaf@math.um.ac.ir\\[2pt]
fa-mo26@stu-mail.um.ac.ir\\
az-ho13@stu-mail.um.ac.ir}
\date{ }
\maketitle

\begin{abstract}
In this paper we present some inequalities for the order, the
exponent and the number of generators of the $c$-nilpotent
multiplier (the Baer invariant with respect to the variety of
nilpotent groups of class at most $c \geq 1$) of a finite group
and specially of a finite $p$-group. Our results generalize some
previous related results of M.R. Jones and M.R.R. Moghaddam. Also,
we
show that our results improve some of the previous inequalities.
\end{abstract}

\section{Introduction and Motivation}

 Let $ G $ be a group with a presentation $F/R $, where $ F $ is a
free group. Then the Baer invariant of $G$ with respect to the
variety $ {\mathcal V}$, denoted by ${\mathcal V} M(G)$, is
defined to be
$${\mathcal V} M(G)=\frac{R \cap V(F)}{[R V^* F]},$$
where $V(F)$ is the verbal subgroup of $ F $  and $$ [R
V^{*}F]=\langle v(f_1,...,f_{i-1},f_{i}r,f_{i+1},...,f_n)
 v(f_1,...,f_{i},...,f_n)^{-1} |$$ $$ \ \ \  r\in R, f_i\in F , v
\in V, 1\leq i\leq n, n \in N \rangle.$$ One can see that the Bear
invariant of a group $ G $  is always abelian and independent of
the choice of the presentation of  $ G $. In particular, if
$\mathcal{V}$ is the variety of abelian groups, $\mathcal{A}$,
then the Baer invariant of the group $ G $  will be $$ \frac{R
\cap F'}{[R, F]},$$which is isomorphic to the Schur multiplier of
$ G $ , denoted by $ M(G) $. Also, if $\mathcal{V}$ is the variety
of nilpotent groups of class at most $c \geq 1$, ${\mathcal
N}_{c}$, then the Baer invariant of the group $ G $  will be
$$ {\mathcal
N}_{c}M(G)=\frac{R \cap \gamma_{c+1}(F)}{[R,\ _cF]}.$$ We also
call it the  $ c $-nilpotent  multiplier of $ G $, and denote it
by $ M ^{(c)} (G) $ (see [8,10] for further details).

\begin{definition}
A variety $\mathcal{V}$ is said to be a \textit{Schur-Baer
variety} if  for any group $G$ for which the marginal factor
group $G/V^*(G)$ is finite, then the verbal subgroup $V(G)$ is
also finite and $|V(G)|$ divides a power of $|G/V^*(G)|$. Schur
[8] proved that the variety of abelian groups, $\mathcal{A}$, is a
Schur-Baer variety. Also, Baer [1] proved that if $u$ and $v$
have Schur-Baer property, then the variety defined by the word
$[u,v]$ has the above property.
\end{definition}

The following theorem gives a very important property of
Schur-Baer varieties.

\begin{theorem}
([10]) The following conditions on the variety $\mathcal{V}$ are
equivalent: \\
(i) $\mathcal{V}$ is a Schur-Baer variety.\\
(ii) For every finite group $ G $, its Baer invariant,
$\mathcal{V}$$M(G)$, is of order dividing a power of $|G|$.
\end{theorem}

\begin{definition}
([4]) The basic commutators on letters $x_{1}, x_{2}, ...,x_{n},
...$ are defined
as follows: \\
(i) The letters $x_{1}, x_{2}, ...,x_{n}, ...$ are basic
commutators of weight one, ordered by setting $x_{i} < x_{j}$ if
$i<j$.\\
(ii) If basic commutators $c_{i}$ of weight $w(c_{i})< k$ are
defined and ordered, then define basic commutators of weight $ k $
by the following rules. $[c_{i}, c_{j}]$ is a basic commutator of
weight $ k $  if and only if\\
1. $w(c_{i}) + w(c_{j})=k ;$\\
2. $c_{i} > c_{j};$\\
3. If $c_{i}=[c_{s}, c_{t}]$, then $c_{j} \geq c_{t}$.\\
Then we will continue the order by setting $c \geq c_{i}$,
whenever $w(c) \geq w(c_{i}),$ fixing any order among those of
weight $ k $, and finally numbering them in order.
\end{definition}

\begin{theorem}
(P. Hall [4]) Let $F$ be the free group on $\{x_{1}, x_{2},
...,x_{d}\}$, then for all $1\leq i \leq n$,
$$ \frac{\gamma_{n}(F)}{\gamma_{n+i}(F)} $$ is the  free abelian
group freely generated by the basic commutators of weights $ n,
n+1,..., n+i-1 $  on the letters $\{x_{1}, x_{2}, ...,x_{d}\}$.
\end{theorem}

\begin{theorem}
(Witt formula [4]) The number of basic commutators of weight $n$
on $ d $  generators is given by the following formula, $$
\chi_{n}(d)= \frac{1}{n}\sum_{m|n}\mu(m) d^{\frac{n}{m}},$$ where
$\mu(m)$ is the Mobious function, which defined to be
 $$\mu(m)=\left\{\begin{array}{ll}
 1& ; m=1,\\
 0& ; m=p_1^{\alpha_1}...p_k^{\alpha_k} \ \ \ , \exists \alpha_{i} > 1\\
(-1)^s& ; m=p_1...p_s, \end{array}\right.$$ where the $p_i$ are
distinct prime numbers.
\end{theorem}

 M.R. Jones in a series of three papers [5,6,7] studied on the order, the exponent and
the number of generators of the Schur multiplier of finite groups,
specially finite $p$-groups and presented some interesting
inequalities about them. Also M.R.R. Moghaddam [13,14] generalized
some of his results. The following are some of them which we deal
with in this article.

\begin{theorem}
(M.R. Jones [5]) Let $ G $ be a $p$-group of order $p^{n}$ with
center of exponent $p^{k}$. Then $|G'||M(G)|$ is no more than
$p^{\frac{1}{2}(n-k)(n+k-1)}$. In particular
$$|G'||M(G)| \leq p^{\frac{1}{2}n(n-1)}.$$
\end{theorem}

\begin{theorem}
(M.R. Jones [6]) Let $ G $ be a finite group and $ K $ any normal
subgroup of it. Set $
H=G/K $, then\\
(i)   $|M(H)|$ divides $|M(G)||G'\cap K|;$\\
(ii)  $exp(M(H))$ divides $exp(M(G))$ $ exp(G' \cap K);$\\
(iii) $d(M(H)) \leq d(M(G)) + d(G' \cap K).$
\end{theorem}

\begin{corollary}
(M.R. Jones [6]) Let $G $ be a finite $ d $-generator group of
order $p^{n}$. Then
$$p^{\frac{1}{2}d(d-1)} \leq |G'||M(G)| \leq
p^{\frac{1}{2}n(n-1)}.$$
\end{corollary}

In 1981 M.R.R. Moghaddam [13,14] gave varietal generalizations of
Theorem 7 and corollary 8. He presented the following theorem
without the condition of being Schur-Baer on the variety. But this
condition seems to be necessary, because ${\mathcal V}M(G)$ is
finite if and only if variety ${\mathcal V}$ is a Schur-Baer
variety.

\begin{theorem}
(M.R.R. Moghaddam [13]) Let ${\mathcal V}$ be a Schur-Baer variety
and G be a finite group
with a normal subgroup K. Let $H=G/K$, then\\
(i)  $|{\mathcal V}M(H)|$ divides $|{\mathcal V}M(G)||V(G)\cap K|;$\\
(ii) $exp({\mathcal V}M(H))$ divides $exp({\mathcal V}M(G))$ $ exp(V(G) \cap K);$\\
(iii) $d({\mathcal V}M(H))\leq d({\mathcal V}M(G)) + d(V(G) \cap
K)$.
\end{theorem}

\begin{corollary}
(M.R.R. Moghaddam [13,14]) Let ${\mathcal V}$ be the variety of
polynilpotent groups of a given class row. Let $ G $ be a finite
$ d $-generator group of order $p^n$. Then
$$|{\mathcal V}M( \textbf{Z}_{p}^{(d)})| \leq |{\mathcal
V}M(G)||V(G)|\leq|{\mathcal V}M(\textbf{Z}_{p}^{(n)})|,$$
 where
$\textbf{Z}_{n}^{(m)}$ denotes the direct sum of $ m $ copies of
$\textbf{Z}_{n}$.
\end{corollary}

 The following theorem is useful in our investigation.

\begin{theorem}
(B. Mashayekhy and M.R.R. Moghaddam [11]) Let $G \cong
\textbf{Z}_{n_1}\oplus \textbf{Z}_{n_2}\oplus ...\oplus
\textbf{Z}_{n_k}$ be a finite abelian groups, where $n_{i+1}\mid
n_i$ for all $1\leq i\leq k-1$. Then for all $c \geq 1$, the $ c
$-nilpotent multiplier of $ G $ is
$$M^{(c)} (G)= \textbf{Z}_{n_2}^{(\chi _{c+1}(2))} \oplus \textbf{Z}_{n_3}^{(\chi
_{c+1}(3)-\chi _{c+1}(2))} \oplus ... \oplus
\textbf{Z}_{n_k}^{(\chi _{c+1}(k)-\chi _{c+1}(k-1))}.$$
\end{theorem}

A useful corollary can be obtained by using Corollary 10 and
Theorem 11.

\begin{corollary}
Let $G$ be a finite $d$-generator $p$-group of order $p^n$, then
$$ p^{\chi_{c+1}(d)}\leq |M^{(c)}(G)||\gamma_{c+1}(G)|\leq
p^{\chi_{c+1}(n)}.$$
\end{corollary}

The next theorem gives some upper bounds in terms of normal
 subgroups and factor groups.

\begin{theorem}
(M.R. Jones [6]) Let $ G $ be a finite group and $K$ be a central
subgroup of $G$. Set
$H=G/K $, then \\
(i)  $|M(G)||G'\cap K|$ divides $|M(H)||M(K)||H \otimes K|;$\\
(ii) $exp(M(G))$ divides $exp(M(H))exp(M(K))exp(H \otimes K);$\\
(iii) $d(M(G))\leq d(M(H))+ d(M(K)) + d(H \otimes K).$
\end{theorem}

The following theorem is a generalization of the above theorem,
for which the condition of Schur-Baer variety seems to be
necessary too.

\begin{theorem}
(M.R.R. Moghaddam [13]) Let ${\mathcal V}$ be a Schur-Baer variety
and $G $ be a finite group with a marginal subgroup $K$. Let $
G=F/R $ be a free
presentation for $G$ and $K=S/R$. Set $H = G/K\cong F/S$. Then\\
(i)  $|{\mathcal V}M(G)||V(G)\cap K|= |{\mathcal V}M(H)||[S V^* F]/[R V^* F]|;$\\
(ii) $exp({\mathcal V}M(G))$ divides $exp({\mathcal V}M(H))exp([S V^* F]/[R V^* F]);$\\
(iii) $d({\mathcal V}M(G)) \leq d({\mathcal V}M(H))+d([S V^* F]/[R
V^* F])$.
\end{theorem}

\begin{theorem}
(M.R. Jones [6]) Let $G$ be a finite nilpotent group of class $ c
$ and let $ Q_{j} =
G /\gamma_{j}(G) $ for $ 2 \leq j \leq c $. Then\\
(i)  $|G'||M(G)|$  divides $ | M(G/G')| \prod ^{c-1} _{j=1}
|\gamma
_{j+1}(G)\otimes Q_{j+1}  |; $ \\
(ii)  $exp( M(G) ) $ divides   $exp( M(G /G')) \prod ^{c-1} _{j=1}
exp(\gamma
_{j+1} (G) \otimes Q_{j+1}  ) ;$ \\
(iii) $d( M(G) ) \leq  d( M(G / G'))+ \sum ^{c-1} _{j=1} d (\gamma
_{j+1} (G) \otimes Q_{j+1} ) .$
\end{theorem}

The above theorem has an interesting corollary for which we need
the following definition.

\begin{definition}
 Let $X$ be any group. We
say that $X$ has a special rank $r(X)$ if every subgroup of $X$
may be generated by $r(X)$ elements and there is at least one
subgroup of $X$ that cannot be generated by fewer than $r(X)$
elements. It is easy to see that $r(X)=d(X)$ for every abelian
group $X$. Also if N is a normal subgroup of $X$ then $ r(X)\leq
r(X/N)+ r(N) $.
\end{definition}

\begin{corollary}
If $ G $ is a finite $p$-group of  class $c$ and special rank $r$
then $$d(M(G))\leq \frac{1}{2} r((2c-1)r-1).$$
\end{corollary}

\begin{corollary}
If $G$ is a finite $p$-group of  class $c$  and exponent $p^e$,
then $$exp(M(G))\leq p^{ec}$$.
\end{corollary}

M. R. Jones in 1974 gave an improvement of Theorem 15 as follows.

\begin{theorem}
(M.R. Jones [7]) Let $G $ be a finite nilpotent group of class $ c
\geq 2 $. Let  $Z _{j} = Z
_{j}(G) $ for all $ 1 \leq j \leq c $, then \\
(i)  $| \gamma _{c} (G) | |M(G)|$ divides $ | M (G/\gamma _{c}(G))
||\gamma _{c} (G) \otimes (G / Z _{c-1}(G) )| $;\\
(ii) $exp (M(G))$ divides $ exp( M(G/\gamma _{c}(G)) ) exp (\gamma
_{c}(G)\otimes  (G / Z _{c-1}(G)));$\\
(iii) $d ( M(G))  \leq d(M(G/\gamma _{c}(G)))+ d (\gamma _{c} (G)
\otimes (G / Z _{c-1}(G))).$
\end{theorem}

\begin{corollary} Let $G$ be a finite $p$-group of  class $c \geq 2$ and
exponent $ p^e$. Then $$exp(M(G))\leq p^{e(c-1)}.$$
\end{corollary}

G. Ellis [3], S. Kayvanfar and M.A. Sanati [9] generalized the
result of M.R. Jones for the exponent of the Schur multiplier of
$G$ (Corollary 20). A result of G. Ellis [3] shows that if $G$ is
a $p$-group of class $k \geq 2$ and exponent $p^e$, then
$exp(M^{(c)}(G)) \leq  p^{e[ \frac{k}{2}]}$, where $[
\frac{k}{2}]$ denotes the smallest integer $n$ such that $n \geq [
\frac{k}{2}]$. Clearly the recent bound sharpens the bound
obtained in 20.\\
S. Kayvanfar and M.A. Sanati [9] proved that $exp(M(G))\leq
exp(G)$ when $G$ is a finite $p$-group of class 3, 4 or 5
and $exp(G)$ satisfies in some conditions. \\

Now, in this paper we are going to concentrate on the Jones
results and trying to generalize some of them. We will give some
upper bounds for the order, the exponent and the number of
generators of the nilpotent multiplier of a finite group,
specially of a finite $p$-group and compare some of them with
previous results. We use the notation $\otimes^{c+1}(B,A)$ for the
tensor product $ B\otimes A \otimes ... \otimes A$ involving $c$
copies of $A$, where $A$ and $B$ are arbitrary groups. The main
results of this article are as follows \\

The following theorem is a generalization of Theorem 6. \\

\textbf{Theorem A.  } Let $G$ be a finite $p$-group of order $ p
^{n} $, $B$ be a cyclic subgroup of $Z(G)$ of order $ p ^{k} $,
where $ p ^{k} = exp (Z(G)) $, and $ A= G/B $ be a
$d$-generator group. Then \\
$$ | \gamma _{c+1}(G) | | M ^{(c)} (G) | \leq p ^{\chi _{c+1}
(n-k)+ dk(1+d) ^{c-1}} .$$
The next theorem extends Theorem 15 and has several interesting corollaries.\\

\textbf{Theorem B.  } Let $G$ be a finite nilpotent group of class
$ t \geq 1 $ and let $ Q_{j} = G/\gamma
_{j}(G) $ for $ 2 \leq j \leq t $. Then\\
(i) $| \gamma _{c+1} (G)|| M^{(c)} (G) |$ divides $ | M^{(c)}
(G/G') | \prod ^{t-1} _{k=1} |\otimes ^{c+1}(\gamma _{k+1} (G),
Q_{k+1}) |; $ \\
(ii)$exp( M^{(c)} (G) ) $ divides $ exp( M^{(c)}(G/G')) \prod
^{t-1}_{k=1} exp (\otimes ^{c+1}(\gamma _{k+1} (G), Q_{k+1})); $ \\
(iii)$d( M^{(c)} (G) ) \leq  d( M^{(c)}(G/G'))+ \sum ^{t-1}
_{k=1}d(\otimes ^{c+1}(\gamma _{k+1} (G), Q_{k+1}) ). $ \\

The following theorem generalizes Theorem 19 and gives some
bounds in terms of the lower central series. \\

\textbf{Theorem C.  } Let $G$ be a finite nilpotent group of class
$ t \geq 2 $. Let  $ Z _{j} = Z _{j}(G) $ for $ 1
\leq j \leq t $, then \\
(i)\

a) If $ c+1 \leq t $, then \

$| \gamma _{t} (G) | | M^{(c)} (G) | $ divides $  | M^{(c)}
(G/\gamma _{t} (G)) |  |\otimes ^{c+1}(\gamma _{t} (G) , G/Z
_{t-1}(G)) . $ \

 b) If $ c+1 > t $, then \

 $ | \gamma _{c+1} (G) | | M^{(c)} (G) | $ divides $  |
M^{(c)}(G/\gamma _{t} (G)) | |\otimes ^{c+1}(\gamma _{t} (G) , G/Z
_{t-1}(G)) |.$ \\
(ii) $exp ( M^{(c)} (G)) $ divides $  exp( M^{(c)}(G/\gamma _{t}
(G))) exp (\otimes ^{c+1}(\gamma _{t} (G), G/Z _{t-1}(G)) ). $\\
(iii) $d ( M^{(c)} (G))  \leq  d( M^{(c)}(G/\gamma _{t} (G)))+ d
(\otimes ^{c+1}(\gamma _{t} (G) , G/Z _{t-1}(G))).$

\section{ Proofs of Main Results}

In order to prove main results we need the following lemmas.

\begin{lemma}
Let $G$ be a group, then for any positive integer $i$
and any normal subgroups $M$, $N$ of $G$, we have
$$ [M, \gamma_{i}(N)]\subseteq[M,\ _{i}N].$$
\end{lemma}

\begin{proof}
We use induction on $ i $. Suppose $ [M, \gamma_{i}(N)] \subseteq
[M,\ _{i}N]$, for any normal subgroups $ M$, $N$ of $G$. Then by
the Three Subgroup Lemma we have
\begin{eqnarray*}
[ \gamma_{i+1}(N), M]=[\gamma_{i}(N),N,M]&\subseteq&
 [N,M,\gamma_{i}(N)][M,\gamma_{i}(N),N] \\& \subseteq& [M,\
 _{i+1}N].
\end{eqnarray*}
\end{proof}

\begin{lemma} Let $F/R$ be a presentation of $G$ as a factor
group of a free group $F$. Let $B=S/R $ be a normal subgroup of
$G$, so that $A=G/B \cong F/S$. Then there exists the following
epimorphism
$$\otimes ^{c+1}(B , A) \longrightarrow  \frac{[S,\ _{c}F]}{[R,\ _{c}F]
[S,\ _{c+1}F] \prod_{i=2}^{c+1}\gamma_{c+1}(S,F)_{i}} ,$$ where
for all $2\leq i\leq c$, $\gamma_{c+1}(S,F)_{i}= [D_1, D_2, ...,
D_{c+1}]$ such that $D_{1}=D_{i}=S $ and $D_{j}=F$, for all $j
\neq 1,i$.
\end{lemma}

\begin{proof}
Define $$\theta: \otimes ^{c+1}(\frac{S}{RS'}, \frac{F}{SF'})
\longrightarrow \frac{[S,\ _cF] }{[R,\ _cF][S,\ _{c+1} F]\prod
^{c+1}_{i=2}\gamma_{c +1}(S,F)_{i}}$$ by $\theta(sRS', f_1 SF',
..., f_c SF') = [s,f_1, ..., f_c][R,\ _cF] [S,\ _{c+1}F]\prod
^{c+1}_{i=2}\gamma_{c +1}(S,F)_{i}$, where $s \in S$ and $f_i \in
F$ for all $i=1, 2, ..., c$. The usual commutator  calculations
and Lemma 21 show that $\theta$ is well defined. Also for any
$s_1,s_2 \in S$ and $f_1, ..., f_c, f'_{1},
..., f'_{c} \in F,$ we have \\
$[s_1 s_2,f_1, ... ,f_c] \equiv [s_1, f_1, ..., f_c][s_2, f_1,
..., f_c] $\\
$[s_1, f_1, ..., f_{i}f'_{i}, ..., f_c]\equiv [s_1, f_1, ...,
f_{i}, ..., f_c][s_1, f_1, ..., f'_i, ..., f_c]( mod [S,\ _{(c+1)}F]).$\\
Then $\theta $ is a multilinear map. Therefore the universal
property of the tensor product completes the proof.
\end{proof}
Now we are ready to prove Theorem A.\\

\textbf{Proof of Theorem A}.\\

 Let $F/R $ be a free presentation of $G$  with $ B= S/R $, so that $ A=G/B=F/S $.
Then $$| \gamma _{c+1}(G) | | M^{(c)}(G)|  = | \frac{\gamma
_{c+1}(F)R}{R} ||\frac{\gamma_{c+1}(F)\cap R}{[R,\ _cF]}|  =
 | \frac{\gamma _{c+1}(F)}{\gamma_{c+1}(F) \cap R} ||\frac{\gamma_{c+1}(F)\cap R}{[R,\ _cF]}|$$

 $$  =
|\frac{\gamma _{c+1}(F)/[R,\ _cF]}{(\gamma_{c+1}(F)\cap R)/[R,\ _
cF]}||\frac{\gamma_{c+1}(F)\cap R}{[R,\ _cF]}|   =
|\frac{\gamma_{c+1}(F)}{[R,\ _cF]}| =|\frac{\gamma_{c+1}(F)}{[S,\
_ cF]}||\frac{[S,\ _cF]}{[R,\ _cF]}|.$$ On the other hand by
corollary 12 we have  $$ |\frac{\gamma _{c+1} (F)}{[S,\ _cF]}|=
|\gamma _{c+1} (A) | |M^{(c)} (A)| \leq p ^{ \chi  _{c+1}
(n-k)}.$$ Now it is enough to show that $ |[S,\ _cF]/[R,\ _cF]|
\leq p ^{dk( 1+d) ^{c-1}}$. We use the following notation for all
$1 \leq j \leq c-1$
 $$ P _{j} = \prod _{( D_1,...,D_{c-1})\in Y_j} [ S,F,D _{1},...,D_{c-1} ], $$ \\
 where $$Y_j=\{(D_1,...,D_{c-1})| \ \ \exists \ i_1,i_2,...,i_j \  s.t. \ D_k =S \
 for\  all\  $$
 $$\ \ \ \ \ \ \ \ \ \ \ \ \ \ \ \ \ \ \ \ k = i_s,   1\leq s \leq
 j \ and\  D_k =F,\   otherwise\  \},$$
and $P_0 = [S,\ _cF]$ , $P_c = \gamma_{c+1}(S)$. It is easy to see that \\
$ |[S,\ _cF]/[R,\ _cF]| = | P _{0} /( [R,\ _cF] P _{1} ) | | ([R,\ _cF] P _{1} ) / [R,\ _cF] |
 = ... $ \\
$ = |P _{0} /( [R,\ _cF] P _{1}) || ( [R,\ _cF] P _{1} ) / ([R,\
_cF] P _{2}) |...| ( [R,\ _cF] P _{c-1} ) / [R,\ _cF] |.$ \\
Since $ B \leq Z (G) $, $ [S,F] \leq R $ and thus $S'\leq R$.
Also, since $M(B)=1$, $S'\bigcap R=[R,S]$. So that $S'=[R,S]$.
Hence $[S,S,\ _{c-1}F]\leq[R,\ _cF]$. Therefore by lemma 22 we
have the following epimorphism

$$\otimes ^{c+1}( B , A)  \rightarrow \frac{ P_{0}
}{[R,\ _cF]P_{1}} $$\\
Also by an argument similar to the proof of lemma 22 one can
deduce the following epimorphism
$$\oplus \sum_{(D_1,...,D_{c-1})\in Y'_j} B \otimes A \otimes  D_1\otimes ... \otimes D_{c-1}
\rightarrow \frac{ [R,\ _cF]P_{j}}{ [R,\ _cF]P_{j+1}},$$where
$$Y'_j=\{ (D_1,...,D_{c-1}) | \ \ \exists \ i_{1},...,i_{j}\  s.t.
\ D_k =B\ for\ all \ $$ $$ \ \ \ \ \ \ \ \ k = i_s,  1\leq s \leq
j
 \ and\  D_k =A ,\   otherwise\ \}, $$
for all $1 \leq j \leq c-1$. Now since $ | A \otimes B | \leq $
min $ \{|A _{ab}|^{d(B)} ,|B_{ab}|^{d(A)} \}$, $ |( [R,\ _cF]
P_{j}) / ([R,\ _cF] P_{j+1} )| \leq p^{ { c-1\choose j-1 }
kd^{c-j+1}}$ for all $ 1 \leq j
\leq c-1 $. Therefore\\
$ | [S,\ _cF] / [R,\ _cF]|\leq p ^{{c-1\choose 0
}kd^{c}+{c-1\choose 1 }kd^{c-1}+ ... +{c-1\choose c-1}kd} $ = $
p^{kd(d+1)^{c-1}}$.\\

The following example compares the above bound and the upper bound
of Corollary 12 .\\

\begin{exam}
 If  $G = {\bf Z}_{p^k}\oplus {\bf Z}_{p^{\alpha_1}}\oplus
{\bf Z}_{p^{\alpha_2}}\oplus ... \oplus {\bf Z}_{p^{\alpha_d}}$
such that $\alpha_d \leq... \leq \alpha_2 \leq \alpha_1 \leq k$
and $k
> 1$, then $B={\bf Z}_{p^k}$, $A \cong {\bf
Z}_{p^{\alpha_1}}\oplus {\bf Z}_{p^{\alpha_2}}\oplus ... \oplus
{\bf Z}_{p^{\alpha_d}}$ and $n=k+t$ where $t=\alpha_1+ \alpha_2 +
\ldots + \alpha_d$. Now for $c=2$, by Corollary 12 $$|\gamma
_{c+1}(G) | | M ^{(c)} (G) | \leq p ^{\chi _{3}(n)} = p^{1/3(t^3
-t +3kt^2 +k^3+3k^2-k)}.$$ But using Theorem A we have $$| \gamma
_{c+1}(G) | | M ^{(c)} (G) | \leq p ^{\chi _{3} (t)+ dk(1+d)
}=p^{1/3(t^3-t+3dk+3d^2k)}\leq p^{1/3(t^3-t+3t^2k+3tk)}.$$ It is
easy to see that $k^3+3k^2-k > 3tk$. Therefore the previous bound
in Corollary 12 is larger than the bound of Theorem A for all
finite abelian $p$-groups but elementary abelian $p$-groups.
\end{exam}

Now the following theorem is needed to prove Theorem B.

\begin{theorem}
Let $ G$ be a finite nilpotent group of class $ t \geq 2 $ , let $
G = F/R $ be a free presentation of $G$ . Then \\
(i) $ | \gamma _{c+1} (G) | | M^{(c)} (G) | = | M^{(c)} (G/G') |
\prod^{t-1}_{k=1}| [\gamma_{k+1} (F)R ,\ _cF ] / [\gamma_{k+2}
(F)R ,\ _cF ] |; $\\
(ii)  $  exp( M^{(c)} (G) ) $ divides $exp( M^{(c)} (G/G') )
\prod^{t-1}_{k=1} exp( [\gamma_{k+1} (F)R ,\ _cF ] / [\gamma_{k+2}
(F)R ,\ _cF ] ). $ \\
(iii) $  d ( M^{(c)} (G) ) \leq d( M^{(c)} (G/G') ) + \sum
^{t-1}_{k=1} d( [\gamma_{k+1} (F)R ,\ _cF ]
 / [\gamma_{k+2}(F)R ,\ _cF ] ); $
\end{theorem}

\begin{proof}
With the previous notation, we have

 (i)\begin{eqnarray*}
  | \gamma _{c+1}(G) | | M^{(c)} (G) | &=& | \gamma _{c+1} (F)/[R,\ _cF]|\\ &=&
|\gamma _{c+1} (F) /( [R \gamma _{2}(F),\ _cF] ) | |[R \gamma _{2}
(F) ,\ _cF ] / [ R,\ _cF] | \\ &=&| M^{(c)} (G/G') || [R \gamma
_{3} (F) ,\ _cF ]/[ R,\ _cF] |\\& & | [R \gamma _{2} (F) ,\ _cF
]/[ R \gamma _{3} (F),\ _cF] |\\ &=&... \\
&=&| M^{(c)} (G/G') || [R\gamma _{t+1} (F) ,\ _cF ]/[ R,\ _cF] |
\\ & & \prod ^{t-1}_{k=1}| [R \gamma _{k+1} (F) ,\ _cF ]/[ R \gamma
_{k+2} (F),\ _cF] |\\
&=&| M^{(c)} (G/G') | \prod ^{t-1}_{k=1} | [R \gamma _{k+1} (F) ,\
_cF ]/[ R \gamma _{k+2}(F),\ _cF] |.
\end{eqnarray*}

(ii)The proof is similar to (i).
\begin{eqnarray*}
 exp ( M^{(c)} (G) ) &=& exp ( R \cap \gamma _{c+1} (F)/[R,\ _cF])  \\
&\mid& exp ( \gamma _{c+1} (F)/[R,\ _cF] ) \\
&\mid& exp(M^{(c)} (G/G') )exp ( [R \gamma _{2} (F) ,\ _cF ]/[ R,\ _cF] ) \\
&\mid& exp(M^{(c)} (G/G') )exp( [R \gamma _{3} (F) ,\ _cF ]/[ R ,\
_cF] )\\
& &exp( [R \gamma _{2} (F) ,\
_cF ]/ [R \gamma _{3} (F) ,\ _cF ])\\
 &\mid& ...\\
 &\mid& exp(M^{(c)} (G/G') ) \prod^{t-1}_{k=1} exp( [\gamma_{k+1} (F)R ,\ _cF
] / [\gamma_{k+2} (F)R ,\ _cF ] ).
\end{eqnarray*}

(iii) We have
\begin{eqnarray*}
  d ( M^{(c)} (G) ) &=& r ( M^{(c)} (G) ) = r ( R
 \cap \gamma _{c+1} (F)/[R,\ _cF])  \\
 &\leq&  r ( \gamma _{c+1} (F)/[R,\ _cF] )\\
&\leq& r ( M^{(c)} (G/G'))+ r ( [R \gamma _{2} (F) ,\ _cF ]/[ R,\ _cF] ) \\
 &=& d ( M^{(c)} (G/G') )+ d ( [R
\gamma _{2} (F) ,\ _cF ]/[ R,\ _cF] ) \\
&\leq& d ( M^{(c)} (G/G') ) + d ( [R \gamma _{3} (F) ,\ _cF ]/[
R,\ _cF] )\\ &+& d( [R \gamma _{2} (F) ,\ _cF ]/[ R \gamma _{3}
(F),\ _cF] )\\
&\leq&...\\
&\leq&  d ( M^{(c)} (G/G') ) + \sum ^{t-1}_{k=1} d( [R \gamma
_{k+1} (F) ,\ _cF ]/[ R \gamma _{k+2}(F),\ _cF] ).
\end{eqnarray*}

\end{proof}

\hspace{-0.65cm}\textbf{Proof of Theorem B}. \\

Using the notation of lemma 22, put $B= (\gamma _{k+1} (F)R) /R $
and $ A= G/\gamma_{k+1}(G)= Q_{k+1} $, then we have the following
epimorphism
$$ \otimes ^{c+1}(\gamma _{k+1}(G) , Q_{k+1})  \rightarrow  \frac{[R \gamma _{k+1} (F),\ _cF]}
{([R,\ _cF][R \gamma _{k+1} (F),\ _{c+1}F] \prod ^{c+1}_{i=2}
\gamma _{c+1}(R \gamma _{k+1} (F),F) _{i})}. $$ On the other hand
 $$ [R,\ _cF] [R \gamma _{k+1} (F) ,\ _{c+1}F] \prod _{i=2}
^{c+1} \gamma _{c+1} (R \gamma _{k+1} (F),F)_i \leq [\gamma
_{k+2} (F)R,\ _cF], $$ since $ [R \gamma _{k+1} (F) ,\ _{c+1}F] =
[R,\ _{c+1}F][\gamma _{k+1} (F),\ _{c+1}F]  \leq [\gamma
_{k+2}(F)R,\ _cF] $. Also, for all positive integers $n$, $m$
such that $ m+n=c-1 $, we have \\
$ [R \gamma _{k+1} (F) ,\ _nF,R \gamma _{k+1} (F),\ _mF ]  $\\
 $=[R,\ _nF,R\gamma _{k+1} (F) ,\ _mF] [\gamma _{k+1} (F),\ _nF,R,\ _mF][\gamma _{k+1}
(F) ,\ _nF, \gamma _{k+1} (F),\ _mF ]  $\\
$ \leq [R,\ _cF][R,\gamma _{k+n+1}(F),\ _mF] \gamma _{2k+c+1} (F)
\leq
[R \gamma _{k+2} (F) ,\ _cF]  $.\\
Hence
$$ \otimes ^{c+1}(\gamma _{k+1}(G) ,Q_{k+1}) \rightarrow  \frac{[R \gamma _{k+1} (F) ,\ _cF]}{ [R \gamma
_{k+2} (F) ,\ _cF]} $$ is an epimorphism. Now the results follows
by theorem 24.\\

The following example shows that the above bound for the order of
$c$-nilpotent multiplier of a finite $p$-group is sometimes
smaller
than the bound which is obtained in corollary 12.\\

\begin{exam}
If $G$ is an extra special $p$-group of order $ p^{3} $  then $G$
is a nilpotent group of class 2. By Theorem B(i) we have
\begin{eqnarray*}
  |\gamma _{3} (G) | | M^{(2)} (G) | &\leq& | M^{(2)} (G/G') |
|\gamma_{2} (G) \otimes (G/\gamma _{2}(G)) \otimes (G/ \gamma
_{2}(G)) | \\ &=& |M^{(2)} (\textbf{Z} _{p} \oplus \textbf{Z} _{p}
)| |\textbf{Z} _{p} \otimes( \textbf{Z} _{p} \oplus \textbf{Z}
_{p} ) \otimes (\textbf{Z} _{p} \oplus \textbf{Z} _{p} ) | =
p^{6}.
\end{eqnarray*}
But Corollary 12 implies that, $ | \gamma
_{3} (G) | | M^{(2)} (G) | \leq p^{\chi _{3} (3)} = p^{8}$ .
\end{exam}

\begin{corollary}
If $G$ is a finite $d$-generator $p$-group
of special rank $r$ and nilpotency class $t$ then $ d(M^{(c)} (G))
\leq \chi _{c+1} (d) + r ^{c+1} (t-1). $
\end{corollary}

\begin{proof}
Since $G$ is a $p$-group and $ d(G) = d (G/G')=d $, $d(
M^{(c)}(G/G' ))= \chi _{c+1} (d)$ by Theorem 11. In addition $ d(
\otimes ^{c+1}(\gamma _{k+1}(G) , Q_{k+1}) ) \leq d( \gamma
_{k+1}(G)) d( Q _{k+1}) ^{c} \leq rd ^{c} \leq r ^{c+1} $. Hence
the required assertion follows by Theorem B for $ t\geq 1 $.
\end{proof}

Note that the inequality in corollary 26 is attained for all
elementary abelian $p$-groups. Now as a final application of
Theorem B we have the following result.

\begin{corollary}
Further to the notation and assumptions of Theorem B, let $ e _{j}
= min \{ exp (Q_{j+1}),\ exp (\gamma _{j+1}(G)) \} $ for $ 1 \leq
j \leq c-1 $. Then $ exp( M^{(c)} (G) ) \leq  exp( G/G')  \prod
^{t-1} _{j=1} e _{j}$. In particular, if G has exponent $ p ^{e}$
then $ exp( M^{(c)} (G) ) \leq p ^{et} $.
\end{corollary}

\begin{proof} Since $G/G' $ is an abelian group, by Theorem
11 $exp( M^{(c)}(G/G')) \leq exp(G/G') ).  $ Also, by the
properties of tensor products, we have\\
$ exp ( \otimes ^{c+1}(\gamma _{j+1}(G) , Q_{j+1})) \leq e _{j} $.
Now the result holds by Theorem B.
\end{proof}

\begin{exam} It can be seen that the inequality of Corollary
27 is attained. Let $G$ be a dihedral group of order 8, $D_8$. By
a theorem of M. R. R. Moghaddam [15] we have
 $$M^{(c)}(D_8) \cong \textbf{Z}_4 \oplus  \underbrace{\textbf{Z}_2 \oplus ...
\oplus \textbf{Z}_2}_{(\chi_{c+1}(2)-1)-times} \ .$$  Then $
exp(M^{(c)}(G)) = 4$.
On the other hand Corollary 27 implies that,  \\
$ exp(M^{(c)}(G))\leq exp(\textbf{Z}_2 \oplus \textbf{Z}_2)\  (min
\{exp (\textbf{Z}_2 \oplus \textbf{Z}_2) ,\ exp(\textbf{Z}_2 ) \})
= 4$. \end{exam}

The following theorem helps us to proof Theorem C.

\begin{theorem}
 Let $G$ be a finite nilpotent group of class $
t \geq 2 $ and let $ G =F/R $ be a free
presentation for $G$. Then \\
(i)\

a) If $ c+1\leq t $, then \

$| \gamma _{t} (G) | | M^{(c)} (G) | =| M^{(c)} (G/\gamma _{t}
(G)) | | [R \gamma _{t}(F),\ _cF ] / [R,\ _cF]|. $ \

b) If $ c+1 > t $, then \

$  | \gamma _{c+1} (G) | | M^{(c)} (G) | =| M^{(c)} (G/\gamma _{t}
(G)) | | [R \gamma _{t}(F),\ _cF ]/ [R,\ _cF]|; $\\
(ii) $ exp   ( M^{(c)} (G))$ divides $exp  ( M^{(c)} (G/\gamma
_{t} (G)))
exp ( [R \gamma _{t}(F),\ _cF ] / [R,\ _cF]); $ \\
(iii)  $ d ( M^{(c)} (G))  \leq  d ( M^{(c)} (G/\gamma _{t} (G)))+
d( [R \gamma _{t}(F),\ _cF ] / [R,\ _cF]). $
\end{theorem}

\begin{proof}
 Since $ \gamma _{t}(G) = (\gamma _{t}(F) R )/ R
\cong \gamma _{t}(F)  /( R \cap \gamma _{t}(F)) $ and $  G /
\gamma _{t}(G) \cong  F/ (\gamma _{t}(F)R) $, we have  $( M^{(c)}
(G/\gamma _{t} (G))) =  (\gamma _{c+1}(F) \cap
\gamma _{t}(F) R ) / [R \gamma _{t}(F),\ _cF ] $ .\\
(i) We consider two cases. \\
Case One: If $ c+1 \leq t, $then $$\frac{( \gamma _{c+1}(F) \cap R
) \gamma  _{t} ( F)}{ [R \gamma _{t}(F),\ _cF ]} \cong \frac{((
\gamma _{c+1}(F) \cap  R ) \gamma  _{t} ( F)) / [ R,\ _cF]}{[R
\gamma _{t}(F),\ _cF ] / [ R,\ _cF] }. $$ Hence
$$ | \frac{ ( \gamma _{c+1}(F) \cap  R ) \gamma  _{t} ( F)}{[
R,\ _cF] } | = |\frac{[R \gamma _{t}(F),\ _cF ]}{[ R,\ _cF] } |
|M^{(c)} (\frac{G}{\gamma _{t} (G)}) |. $$ But
\begin{eqnarray*}
 \frac{(( \gamma _{c+1}(F) \cap  R ) \gamma  _{t} ( F)) / [
R,\ _cF]}{( \gamma _{c+1}(F) \cap  R )/ [R,\ _cF] }  & \cong &
\frac{ ( \gamma _{c+1}(F) \cap  R ) \gamma  _{t} ( F)}{\gamma
_{c+1}(F) \cap  R }\\& \cong & \frac{\gamma_{t}(F)}{\gamma _{t}(F)
\cap R }  \cong  \gamma _{t}(G),
\end{eqnarray*}
so that
\begin{eqnarray*}
 | \frac{ ( \gamma _{c+1}(F) \cap  R ) \gamma  _{t} ( F)}{[
R,\ _cF] } | & = & |\gamma _{t}(G)|| \frac{\gamma _{c+1}(F) \cap
R}{[ R,\ _cF]} |\\ & = & | \gamma _{t} (G) | | M^{(c)} (G) |.
\end{eqnarray*}
Therefore
$$  | \gamma _{t} (G) | | M^{(c)} (G) | =| M^{(c)} (\frac{G}{\gamma
_{t} (G)}) | | \frac{[R \gamma _{t}(F),\ _cF ]}{[R,\ _cF]}| .$$\\
Case Two: If $ c+1 > t $, then we have

$$M^{(c)} (\frac{G}{\gamma _{t} (G)}) \cong \frac {\gamma _{c+1}(F)}
 { [R \gamma _{t}(F),\ _cF ]}  \cong  \frac{ \gamma _{c+1}(F) / [
R,\ _cF]}{[R \gamma _{t}(F),\ _cF ] / [ R,\ _cF] }.$$ Thus
 $$  | \frac {\gamma _{c+1}(F)  }{ [R,\ _cF ]}| = | M^{(c)} (\frac{G}{\gamma
_{t} (G)}) | | \frac{[R \gamma _{t}(F),\ _cF ]}{[R,\ _cF]}|.$$ But
$$ \frac{ \gamma _{c+1}(F) / [R,\ _cF]}{ (\gamma _{c+1}(F) \cap  R ) /
[R,\ _cF] } \cong \frac{ \gamma _{c+1}(F)}{ \gamma _{c+1}(F) \cap
R }  \cong \frac{ \gamma _{c+1}(F)R }{ R }  \cong \gamma
_{c+1}(G).$$
Hence the result follows as for case one. \\
\\
(ii),(iii) Since
$$ M^{(c)} (\frac{G}{\gamma _{t} (G)}) \cong \frac{ (\gamma
_{c+1}(F) \cap \gamma _{t}(F)R) / [ R,\ _cF]}{[R \gamma _{t}(F),\
_cF ] / [ R,\ _cF] } \ $$and
$$ M^{(c)} (G) \cong   \frac{\gamma _{c+1}(F) \cap R }{[R,\ _cF ]} \leq
\frac{\gamma _{c+1}(F) \cap \gamma _{t}(F)R} { [R,\ _cF ]},$$ we
have
$$exp ( M^{(c)} (G)) \mid exp( \frac{  \gamma _{c+1}(F)
\cap R \gamma  _{t} ( F) }{ [ R,\ _cF]}) \mid exp ( M^{(c)}
(\frac{G}{\gamma _{t} (G)})) exp (\frac{ [R \gamma _{t}(F),\ _cF ]
}{ [R,\ _cF]}).$$ On the other hand by Lemma 21, $(( \gamma
_{c+1}(F) \cap R) \gamma  _{t} ( F)) / [ R,\ _cF]$ is an abelian
group, therefore
$$d( M^{(c)} (G))  \leq r (  \frac{ \gamma _{c+1}(F) \cap R
\gamma  _{t} ( F) }{ [ R,\ _cF] }) = d(  \frac{ \gamma _{c+1}(F)
\cap R \gamma  _{t} ( F) }{ [ R,\ _cF] }) $$
 $$ \leq  d ( M^{(c)} (\frac{G}{\gamma _{t} (G)}))+
d(\frac{[R \gamma _{t}(F),\ _cF]}{ [R,\ _cF]}).$$ This completes
the proof.
\end{proof}
\textbf{Proof of Theorem C}\\

Let $F/R$ be a free presentation for $G$ with $ Z _{j} = Y _{j} /R
$ for $ 1 \leq j \leq t $. Then we have $ \gamma _{t} (G) = \gamma
_{t} (F)R/R $ and $ G/Z _{j} \cong
F/Y _{j} $ for $ 1 \leq j \leq t $. Define \\
$$ \theta : \frac{\gamma _{t} (F)R}{R} \times \frac{F}{Y _{t-1}}
\times ... \times \frac{F}{Y _{t-1}} \longrightarrow \frac{[R \gamma _{t}(F),\ _cF ]}{[R,\ _cF]}$$ \\
by   $ \theta ( gR,f_{1}Y _{t-1},...,f_{c}Y _{t-1}) =
[g,f_{1},...,f_{c} ][R,\ _cF] $ for $ f_{1},...,f_{c} $ in $F$ and
$g$ in $\gamma _{t}(F)$. Suppose $g'=gr$ and $ f ' _{i}= f_{i}y
_{i} $ for $r$ in $R$ and $ y_{i} $ in $ Y _{t-1}$ for $1 \leq i
\leq c $. Then the commutator calculations and Lemma 21 show that
$ [g,f_{1},...,f_{c}] \equiv [g ' ,f '_{1},...,f '_{c}]$ (mod
$[R,\ _cF ]$) and $ \theta $ is well defined. Moreover for $g$, $g
'$ in $ \gamma _{t}(G)$ and $ f _{i}, f ' _{i}$ in $F$ ,
$[gg',f_{1},...,f_{c}] \equiv [g,f_{1},...,f_{c}] [g
',f_{1},...,f_{c}]$ and $ [g,f_{1},...,f_{i}f '_{i}
,...,f_{c}]\equiv [g,f_{1},...,f_{i},..,f_{c}][g,f_{1},...,f
'_{i},..,f_{c}]$ (mod $[R,\ _cF] $) for $ 1 \leq i \leq c $. Hence
$\theta $ is multilinear map and therefore $ ([R \gamma _{t}(F),\
_cF ])/[R,\ _cF]$ is a homomorphic image of $ \otimes
^{c+1}(\gamma _{t} (G), G/Z _{t-1}(G))  $ , by the universal
property of tensor product.
Now the result follows from previous theorem.  \\

\textbf{Remarks}

(i) The inequality in Theorem C (i) is attained for extra special
$p$-group of order $ p^{3} $
 of exponent $p$ and $ c=1 $. Also equality holds in (ii), for dihedral group of order 8. In
 addition, Example 25 helps us to see that the bound is some times better than the bound obtained
 by Corollary 12  for the order of $c$-nilpotent multiplier of a finite
 $p$-group.

(ii) Note that a result similar to Theorem C has been proved in a
different method by J. Burns and G. Ellis [2, Proposition 5].
Their proof is based on nonabelian tensor product argument. Note
that when we consider the exterior product of abelian groups by
the canonical homomorphisms $ \gamma _{t}(G)\rightarrow
G/Z_{t-1}(G)$ and $ G/Z_{t-1}(G)\rightarrow G/Z_{t-1}(G) $ as
crossed modules, then the rule of $\theta$ in the proof of the
Theorem C gives the following epimorphism:
 $$\hat{\theta}: \gamma _{t}(G)\wedge \frac{G}{Z_{t-1}(G)}\wedge ...
 \wedge \frac{G}{Z_{t-1}(G)}\rightarrow \frac{[R \gamma _{t}(F),\ _cF]}{[R,\ _cF]}.$$
 Hence we can replace $\gamma _{t}(G)\otimes G/Z_{t-1}(G)\otimes ...\otimes
 G/Z_{t-1}(G)$ by
 $ \gamma _{t}(G)\wedge G/Z_{t-1}(G)\wedge ...\wedge G/Z_{t-1(G)}$ in Theorem C.
  Also it seems that there is a missing point in the proof of the similar result of J.
Burns and G. Ellis [2, Proposition 5] for the right exactness of
the sequence
 $$ M^{(c)}(G)\rightarrow M^{(c)}(\frac{G}{\gamma _{t}(G)})\rightarrow \gamma _{t}(G)
 \rightarrow 1.$$
So we should state part (i) of Theorem C in two cases.

\end{document}